\documentclass[letterpaper, conference, twocolumn, final]{ieeeconf}

\IEEEoverridecommandlockouts 

\usepackage{url}
\usepackage{cite}
\usepackage{epsf}
\usepackage{times}
\usepackage{graphicx}
\usepackage[english]{babel}
\usepackage{color}
\usepackage{amssymb}  %
\usepackage{subfigure}  %
\usepackage{bm}
\usepackage{amsmath}
\usepackage{makeidx}  
\usepackage{times}
\usepackage{algorithm,algorithmic,xspace}

\usepackage{verbatim}

\DeclareMathOperator*{\argmax}{arg\,max}

\renewcommand{\natural}{{\mathbb{N}}}

\newcommand{\integernonnegative}{{\mathbb{Z}_{\geq0}}}
\newcommand{\naturalzero}{\mathbb{N}_0}
\renewcommand{\naturalzero}{\integernonnegative}

\newcommand{\real}{{\mathbb{R}}}
\newcommand{\subscr}[2]{{#1}_{\textup{#2}}}
\newcommand{\union}{\cup}

\newcommand{\ceil}[1]{\left\lceil #1\right\rceil}

\newcommand{\smallfrac}[2]{\frac{{\scriptstyle #1}}{{\scriptstyle #2}}}

\newcommand{\until}[1]{\{1,\dots,#1\}}

\newcommand{\diag}{\operatorname{diag}}

\newcommand{\CC}{\mathcal{C}}

\newcommand{\subj}{\text{subj to}}

\newcommand{\half}{\frac{1}{2}}

\newcommand{\Vc}{\mathcal{V}_{c}}
\newcommand{\Ec}{\mathcal{E}_{c}}

\newcommand{\Gc}{\mathcal{G}_{c}}

\newcommand{\outnbrs}{\mathcal{N}^O}
\newcommand{\innbrs}{\mathcal{N}^I}

\newcommand{\CS}{\mathrm{\CC}}
\newcommand{\rad}{\mathrm{r}}
\newcommand{\cent}{\mathrm{c}}

\newcommand{\one}{\textbf{1}}

\newcommand{\dcore}{\ceil{\frac{1}{\epsilon}}}

\newcommand{\coreset}{\texttt{coreset}}

\newcommand{\val}{\mathrm{r^2}}

\newtheorem{theorem}{Theorem}[section]

\newtheorem{lemma}[theorem]{Lemma}
\newtheorem{remark}[theorem]{Remark}

\newtheorem{assumption}[theorem]{Assumption}

\newcommand\oprocendsymbol{\hbox{$\square$}}
\newcommand\oprocend{\relax\ifmmode\else\unskip\hfill\fi\oprocendsymbol}

\renewcommand{\theenumi}{(\roman{enumi})}

\graphicspath{{figs/}}

\title{A core-set approach for distributed quadratic programming\\ in big-data classification}

\author{Giuseppe Notarstefano \thanks{
    Giuseppe Notarstefano is with the Department of Engineering, Universit\`a
    del Salento, via Monteroni, 73100, Lecce, Italy,
    \texttt{giuseppe.notarstefano@unisalento.it.} This result is part of a
    project that has received funding from the European Research Council (ERC)
    under the European Union’s Horizon 2020 research and innovation programme
    (grant agreement No 638992 - OPT4SMART).   } 
  }

\begin{document}

\maketitle 


\begin{abstract}
  A new challenge for learning algorithms in cyber-physical network systems is
  the distributed solution of big-data classification problems, i.e., problems
  in which both the number of training samples and their dimension is high.
  Motivated by several problem set-ups in Machine Learning, in this paper we
  consider a special class of quadratic optimization problems involving a
  ``large'' number of input data, whose dimension is ``big''.  To solve these
  quadratic optimization problems over peer-to-peer networks, we propose an
  asynchronous, distributed algorithm that scales with both the number and the
  dimension of the input data (training samples in the classification
  problem).  The proposed distributed optimization algorithm relies on the
  notion of ``core-set'' which is used in geometric optimization to approximate
  the value function associated to a given set of points with a smaller subset
  of points. By computing local core-sets on a smaller version of the global
  problem and exchanging them with neighbors, the nodes reach consensus on a set
  of active constraints representing an approximate solution for the global
  quadratic program.
\end{abstract}


\section{Introduction}
\label{sec:introduction}

Several learning problems in modern cyber-physical network systems involve a
large number of very-high-dimensional input data. The related research areas go
under the names of \emph{big-data analytics} or \emph{big-data classification}.
%
%
%
From an optimization point of view, the problems arising in this area involve a
large number of constraints and/or local cost functions typically distributed
among computing nodes communicating asynchronously and unreliably. 
An additional challenge arising in big-data classification problems is that not
only the number of constraints and local cost functions is large, but also the
dimension of the decision variable is big and may depend on the number of nodes
in the network.

We organize the literature in two parts. First, we point out some recent works
focusing the attention on \emph{big-data optimization} problems, i.e., problems
in which all the data of the optimization problem are big and cannot be handled
using standard approaches from sequential or even parallel optimization.
The survey paper \cite{cevher2014convex} reviews recent advances in convex
optimization algorithms for big-data, which aim to reduce the computational,
storage, and communications bottlenecks. The role of parallel and distributed
computation frameworks is highlighted.  
In \cite{facchinei2015parallel} big-data, possibly non-convex, optimization
problems are approached by means of a decomposition framework based on
successive approximations of the cost function.
%
In \cite{slavakis2014online} dictionary learning tasks motivate the development
of non-convex and non-smooth optimization algorithms in a big-data context. The
paper develops an online learning framework by jointly leveraging the stochastic
approximation paradigm with first-order acceleration schemes.

Second, we review distributed optimization algorithms applied to learning
problems and highlight their limitations when dealing with big-data problems.
An early reference on peer-to-peer training of Support Vector Machines is
\cite{lu2008distributed}. A distributed training mechanism is proposed in which
multiple servers compute the optimal solution by exchanging support vectors over
a fixed directed graph. The work is a first successful attempt to solve SVM
problems over networks. However, the local memory and computation at each node
does not scale with the problem and data sizes and the graph is time-invariant.
In \cite{forero2010consensus} a distributed Alternating Direction Method of
Multipliers (ADMM) is proposed to solve a linear SVM training problem, while in
\cite{lee2012drsvm} the same problem is solved by means of a random projected
gradient algorithm. 
Both the algorithms are proven to solve the centralized problem (i.e., all the
nodes reach a consensus on the global solution), but again show some
limitations: the graph topology must be (fixed, \cite{forero2010consensus}, and)
undirected, and the algorithms do not scale with the dimension of the training
vector space.  
In \cite{boyd2011distributed} a survey on ADMM algorithms applied to statistical
learning problems is given.
In \cite{varagnolo2013finding}
the problem of exchanging only those measurements that are most informative in a
network SVM problem is investigated. For separable problems an algorithm is
provided to determine if an element in the training set can become a support
vector.
The distributed optimization algorithm proposed in
\cite{notarstefano2011distributed} solves part of these problems: local memory
is scalable and communication can be directed and asynchronous. However, the
dimension of the training vectors is still an issue.

The core-set idea used in this paper was introduced in
\cite{badoiu2002approximate} as a building block for clustering, and refined in
\cite{badoiu2008optimal}. In \cite{tsang2005core} the approach was shown to be
relevant for several learning problems and the algorithm re-stated for such
scenarios. A multi-processor implementation of the core-set approach was
proposed in \cite{lodi2010single}. However, differently from our approach, that
algorithm: (i) is not completely distributed since it involves a coordinator, and
(ii) does not compute a global core-set, but a larger set approximating it.

The main contribution of this paper is twofold. First, we identify a distributed
big-data optimization framework appearing in modern classification problems
arising in cyber-physical network systems. In this framework the problem is
characterized by a large number of input data distributed among computing
processors. The key challenge is that the dimension of each input vector is
very-high, so that standard local updates in distributed optimization cannot be
used. For this big-data scenario, we identify a class of quadratic programs that
model several interesting classification problems as, e.g., training of support
vector machines.
Second, for this class of big-data quadratic optimization problems, we propose a
distributed algorithm that solves the problem up to an arbitrary $\epsilon$
tolerance and scales both with the number and the dimension of the input
vectors. The algorithm is based on the notion of core-set used in geometric
optimization to approximate the value function of a given set of points with a
smaller subset of points. From an optimization point of view, a subset of active
constraints is identified, whose number depends only on the tolerance
$\epsilon$. The resulting approximate solution is such that an
$\epsilon$-relaxation of the constraints guarantees no constraint violation.

The paper is organized as follows. In Section~\ref{sec:distrib_optim_framework}
we introduce the distributed optimization problem addressed in the paper and
describe the network model. Section~\ref{sec:distributed_SVM} motivates the
problem set-up by showing a class of learning problems that can be cast in this
set-up. In Section~\ref{sec:core-set consensus} the core-set consensus algorithm
is introduced and analyzed. Finally, in Section~\ref{sec:simulations} a
numerical example is given to show the algorithm correctness.

\section{Distributed quadratic programming framework}
\label{sec:distrib_optim_framework}
In this section we introduce the problem set-up considered in the paper. 
We recall that we will deal with optimization problems in which both the number
of constraints and decision variables are ``big''.

We consider a set of processors $\{1,\ldots,N\}$, each equipped with
communication and computation capabilities.  Each processor $i$ has knowledge of
a vector $s_i\in\real^d$ and needs to cooperatively solve the quadratic program 
\begin{align}
  \notag  \min_{z\in\real^d, r\in\real} &\;\; r^2\\ 
 \subj        &\;\; \| z - s_i\|^2  \leq r^2, \quad i\in\until{N}.  
\label{eq:QP_probl_primal}
\end{align}
The above quadratic program is known in geometric optimization as \emph{minimum
  enclosing ball} problem, since it computes the center of the ball with minimum
radius enclosing the set of points $s_1, \ldots, s_N$.

By applying standard duality arguments, it can be shown that solving
\eqref{eq:QP_probl_primal} is equivalent to solving its dual
\begin{align}
\max_{x_1,\ldots,x_N} &\; \sum_{i=1}^N s_i^Ts_i x_i - \sum_{i=1}^N\sum_{j=1}^N s_i^Ts_j x_i x_j \nonumber\\
\subj &\;  \sum_{i=1}^N x_i = 1 \nonumber\\
         &\;\;   x_i \geq 0 \quad i\in\until{N},
\label{eq:QP_probl}
\end{align}
with $x_i\in\real$, $i\in\until{N}$. The problem can be written in a more
compact form as
\begin{align}
\max_{x\in\real^N} &\; \diag(S^TS) x  - x^T S^T S x \nonumber\\
\subj &\;\;  \one^T x = 1 \nonumber\\
         &\;\;   x \geq 0
\label{eq:QP_probl_vec}
\end{align}
where $S = [ s_1 \; \ldots \; s_N] \in\real^{d\times N}$, $\diag(S^TS)$ is the
vector with elements $s_i^Ts_i$, $i\in\until{N}$, $\one = [1
\; \ldots \; 1]^T \in\real^N$ and $x \geq 0$ is meant component-wise.

We will show in the next sections that this class of quadratic programs arises in
many important big-data classification problems. 

Each node has computation capabilities meaning that it can run a routine to
solve a local optimization problem. Since the dimension $d$ can be big, the
distributed optimization algorithm to solve problem \eqref{eq:QP_probl_primal}
needs to be designed so that the local routine at each node scales ``nicely''
with $d$.

The communication among the processors is modeled by a time-varying, directed
graph (digraph) $\Gc(t)=(\Vc,\Ec(t))$, where $t \in \naturalzero$ represents a
slotted universal time, the node set $\Vc = \{1,\ldots,N\}$ is the set of
processor identifiers, and the edge set $\Ec(t) \subset \{1,\ldots,N\}^{2}$
characterizes the communication among the processors.
Specifically, at time $t$ there is an edge from node $i$ to node $j$ if and only
if processor $i$ transmits information to processor $j$ at time $t$.
The time-varying set of outgoing (incoming) \emph{neighbors} of node $i$ at time
$t$, i.e., the set of nodes to (from) which there are edges from (to) $i$ at
time $t$, is denoted by $\outnbrs_i(t)$ ($\innbrs_i(t)$).
A static digraph is said to be \emph{strongly connected} if for every pair of
nodes $(i,j)$ there exists a path of directed edges that goes from $i$ to $j$.
For the time-varying communication graph we rely on the concept of a jointly
strongly connected graph.

\begin{assumption}[Joint Strong Connectivity] \label{ass:PeriodicConnectivity}
  For every time instant $t \in \mathbb{N}$, the union digraph
  $\Gc^{\infty}(t):=\cup_{\tau = t}^{\infty} \Gc(\tau)$ is strongly
  connected. \oprocend 
\end{assumption}

It is worth noting that joint strong connectivity of the directed communication
graph is a fairly weak assumption (it just requires persistent spreading of
information) for solving a distributed optimization problem, and naturally embeds
an asynchronous scenario.

We want to stress once more that in our paper all the nodes are peers, i.e.,
they run the same local instance of the distributed algorithm, and no node can
take any special role. Consistently, we allow nodes to be asynchronous, i.e.,
nodes can perform the same computation at different speed, and communication can
be unreliable and happen without a common clock (the time $t$ is a universal
time that does not need to be known by the nodes).

\section{Distributed big-data classification}
\label{sec:distributed_SVM}
In this section we present a distributed set up for some fundamental
classification problems and show, following \cite{tsang2005core}, how they can
be cast into the distributed quadratic programming framework introduced in the
previous section.

We consider classification problems to be solved in a distributed way by a
network of processors following the model in
Section~\ref{sec:distrib_optim_framework}. Each node in the network is assigned
a subset of input vectors and the goal for the processors is to cooperatively
agree on the optimal classifier without the help of any central coordinator.

\subsection{Training of Support Vector Machines (SVMs)}
\label{subsec:SVM model}
Informally, the SVM training problem can be summarized as follows. Given a set
of positively and negatively labeled points in a $k$-dimensional space, find a
hyperplane separating ``positive'' and ``negative'' points with the maximal
separation from all the data points. 
The labeled points are commonly called \emph{examples} or
\emph{training vectors}.

Linear separability of the training vectors is usually a strong assumption. In
many important concrete scenarios the training data cannot be separated by
simply using a linear function (a hyperplane). To handle the nonlinear
separability, nonlinear kernel functions are used to map the training samples
into a feature space in which the resulting features can be linearly separated.
That is, given a set of points $p_1, \ldots,p_m \in \real^k$ in the \emph{input
  space} they are mapped into a \emph{feature space} through a function
$p_i\mapsto \varphi(p_i) \in\real^d$. The key aspect in SVM is that $\varphi$
does not need to be known, but all the computations can be done through a so
called \emph{Kernel function} $K$ satisfying $K(p_i,p_j) :=
\varphi(p_i)^T\varphi(p_j)$.

\begin{remark}
  It is worth noting that the dimension of the feature space can be much higher
  than the one of the input space, even infinite (e.g., Gaussian
  kernels). \oprocend
\end{remark}

Following \cite{keerthi2000fast} and \cite{tsang2005core} we will adopt the
following common assumption in SVM. For any $p_i$ in the input space
\begin{align}
  \label{eq:const_kernel}
  K(p_i,p_i) = c, 
\end{align}
with $c$ independent of $i$.
This condition is satisfied by the most common kernel functions used in
SVM as, e.g., the isotropic kernel (e.g., Gaussian kernel), the dot-product
kernel with normalized inputs or any normalized kernel.

For fixed $d\in\natural$, let $\varphi(p_i) \in\real^d$, $i\in\until{N}$, be a
set of $N\!\in\!\natural$ feature-points with associated label $\ell_i \in
\{-1,+1\}$. The training vectors are said to be linearly separable if there
exist $w\in \real^d$ and $b\in\real$ such that $\ell_i (w^T \varphi(p_i) + b) \geq 1$ for
all $i\in\until{N}$. The \emph{hard-margin SVM training} problem consists of finding the
optimal hyperplane $w^T_o x + b_o = 0$, $x\in\real^d$, ($w_o$ is a vector
orthogonal to the hyperplane and $b_o$ is a bias) that linearly separates the
training vectors with maximal margin, that is, such that the distance
\[
\rho(w,b) = \min_{\varphi(p_i) | \ell_i=1} \frac{w^T \varphi(p_i)}{|w|} - \max_{\varphi(p_i) | \ell_i=-1}
\frac{w^T \varphi(p_i)}{|w|}
\]
is maximized.  Combining the above equations it follows easily that
$\rho(w_o,b_o) = 2/|w_o|$. Thus the SVM training problem may be written as a
quadratic program
\begin{equation}
  \begin{split}
    \min_{b,w} & \enspace {\smallfrac{1}{2}} \|w\|^2 \\
    \subj & \enspace \ell_i (w^T \varphi(p_i) + b) \geq 1 \qquad i\in\until{N}.
  \end{split}
  \label{eq:SVM_probl}
\end{equation}

In most concrete applications the training data cannot be separated without
outliers (or training errors).
A convex program that approximates the above problem was introduced in
\cite{CC-VV:95}. The idea is to introduce positive slack variables in order to
relax the constraints and add an additional penalty in the cost function to
weight them. 
The resulting classification problems are known as \emph{soft marging problems}
and the solution is called \emph{soft margin hyperplane}.

Next, we will concentrate on a widely used soft-margin problem, the
\emph{$2$-norm problem}, which adopts a quadratic penalty function. 
Following \cite{tsang2005core}, 
we will show that its dual version is a quadratic program with the structure of
\eqref{eq:QP_probl}.  The $2$-norm optimization problem turns out to be
\begin{equation}
  \begin{split}
    \min_{w,b, \rho, \xi_1,\ldots,\xi_N} & \;\; \smallfrac{1}{2} \|w\|^2
    +\smallfrac{1}{2} b^2 - \rho + \smallfrac{C}{2} \, \sum_{i=1}^N \xi_i^2 \\
    \subj &  \;\; \ell_i (w^T \varphi(p_i) + b) \geq \rho-\xi_i \qquad i\in\until{N}.\\
  \end{split}
  \label{eq:2norm_SVM_probl}
\end{equation}
Solving problem \eqref{eq:2norm_SVM_probl} is equivalent to solving the dual problem
\begin{align}
   \notag  \max_{x_1,\ldots,x_N} &  - \half
    \sum_{i=1}^N \sum_{j=1}^N x_i
    x_j \!\left(\ell_i \ell_j \varphi(p_i)^T \varphi(p_j) + \ell_i \ell_j +
      \frac{\delta_{ij}}{C}\right)\\ 
\notag    \subj  & \; \sum_{i=1}^N x_i = 1\\
    & \; x_i \geq 0 \qquad i \in \until{N},
  \label{eq:2norm_SVM_probl_dual}
\end{align}
where $\delta_{ij} \!\!=\!\! 1$ if $i\!\!=\!\!j$ and $\delta_{ij} \!\!=\! 0$ otherwise. 
\begin{remark}[Support vectors]
  The vector $w_o$ defining the optimal hyperplane can be written as linear
  combination of training vectors, $w_o = \sum_{i=1}^{N} \ell_i x_i \varphi(p_i)$,
  where $x_i\geq 0$ and $x_i>0$ only for vectors satisfying $\ell_i(w^T
  \varphi(p_i) + b)=\rho-\xi_i$. These vectors are called \emph{support vectors}.
  Support vectors are basically active constraints of the quadratic program. \oprocend
\end{remark}

Now, we can notice that defining $\tilde{K}(p_i,p_j) = \left(\ell_i \ell_j
  \varphi(p_i)^T \varphi(p_j) + \ell_i \ell_j + \frac{\delta_{ij}}{C}\right)$,
it holds 
\[
\tilde{K}(p_i,p_i) = c + 1 + \frac{1}{C},
\]
so that the constant term
$\smallfrac{1}{2}\sum_{i=1}^N x_i \tilde{K}(p_i,p_i) =
\smallfrac{1}{2}\sum_{i=1}^N x_i \left(\ell_i \ell_i \varphi(p_i)^T \varphi(p_i)
  + \ell_i \ell_i + \frac{\delta_{ii}}{C}\right)$,
can be added to the cost function. Thus, posing
\[
\widetilde{\varphi}(p_i) =
\begin{bmatrix}
\ell_i  \varphi(p_i)\\
\ell_i\\
\frac{1}{\sqrt{C}} e_i
\end{bmatrix}, 
\]
with $e_i$ the $i$th canonical vector (e.g., $e_1 = [1\; 0\; \ldots 0]^T$),
problem \eqref{eq:2norm_SVM_probl_dual} can be equivalently rewritten as 
\begin{align*}
\max_{x_1,\ldots,x_N} &\; \sum_{i=1}^N \tilde{\varphi}(p_i)^T
                        \tilde{\varphi}(p_i) x_i - \sum_{i=1}^N\sum_{j=1}^N
                        \tilde{\varphi}(p_i) \tilde{\varphi}(p_j) x_i x_j
                        \notag\\ 
\subj &\;  \sum_{i=1}^N x_i = 1 \notag\\
         &\;\;   x_i \geq 0 \quad i\in\until{N},
\end{align*}
which has exactly the same structure as problem~\eqref{eq:QP_probl}.

It is worth noting that even if the dimension $d$ of the training samples in the
feature space ($\varphi(p_i)\in\real^d$) is small compared to the number of samples
$N$ (so that in problem \eqref{eq:SVM_probl} the dimension of the decision
variable is much smaller than the number of constraints), in the ``augmented''
soft-margin problem \eqref{eq:QP_probl} we have $\widetilde{\varphi}(p_i)\in\real^{d+1+N}$. Thus, in
the primal problem \eqref{eq:QP_probl_primal} the dimension of the decision
variable is of the same order as the number of constraints.

\subsection{Unsupervised classification and clustering}
\label{subsec:unsupervised}
Next, we recall from \cite{tsang2005core} that also some unsupervised soft
margin classification problems can be cast into the same problem set-up of the
paper. 

First, from \cite{tsang2005core} and references therein, it can be shown that
problem \eqref{eq:QP_probl_primal} is equivalent to the hard-margin Support
Vector Data Description (SVDD) problem. Indeed, given a kernel function $K$ and
feature map $\varphi$, the hard-margin SVDD primal problem is
\begin{align}
  \notag  \min_{z\in\real^d, r\in\real} &\;\; r^2\\ 
 \subj        &\;\; \| z - \varphi(p_i)\|^2  \leq r^2, \quad i\in\until{N}.  
\label{eq:SVDD_probl_primal}
\end{align}
In other words, problem \eqref{eq:SVDD_probl_primal} is simply problem
\eqref{eq:QP_probl_primal} in the feature space. 

Another unsupervised learning problem that can be cast into the problem set-up
\eqref{eq:QP_probl} is the so called \emph{one-class L2 SVM},
\cite{tsang2005core}. Given a set of unlabeled input vectors the goal is to
separate outliers from normal data. From an optimization point of view, the
problem can be written as problem \eqref{eq:2norm_SVM_probl}, but with $b=0$ and
$\ell_i = 1$ for all $i\in\until{N}$.  Thus, using the same arguments as in the
previous subsection, the problem can be rewritten in the form
\eqref{eq:QP_probl}.

To conclude this motivating section, we recall from \cite{badoiu2002approximate}
that algorithms solving problem \eqref{eq:QP_probl_primal} are important
building blocks for clustering problems.

\section{Core-set consensus algorithm}
\label{sec:core-set consensus}
In this section we introduce the core-set consensus algorithm to solve
problem~\eqref{eq:QP_probl_primal} (or equivalently its dual
\eqref{eq:QP_probl}) in a distributed way. We start by introducing the notion of
core-set borrowed from geometric optimization and a routine from
\cite{badoiu2008optimal} that is proven to compute an efficient core-set for
\eqref{eq:QP_probl_primal}, which in geometric optimization is known as
\emph{minimum enclosing ball problem}.

\subsection{Core sets: definition and preliminaries}
\label{subsec:core-sets prelim}
In the following we will a little abuse notation by denoting with
$G\in\real^{d\times m}$ both the $d\times m$ matrix and the set of $m$ vectors
(or points) of dimension $d$.
Let $S \in \real^{d\times N}$ be a matrix of ``points'' $s_i\in\real^d$,
$i\in\until{N}$, (i.e., a matrix in which each column represents a vector in
$\real^d$) with $\cent(S)$ and $\rad(S)$ respectively the center and radius of
the minimum enclosing ball containing the points of $S$. We say that $\CS
\subset S$ is an \emph{$\epsilon$-core-set} for the Minimum Enclosing Ball (MEB)
problem, \eqref{eq:QP_probl_primal}, if all the points of $S$ are at distance at
most $(1+\epsilon) \rad(S)$ from the center $\cent(\CS)$ of the minimum
enclosing ball containing $\CS$. 
Note that $\rad(S)^2 = r_*^2$, with $r_*$ being the optimal value of
\eqref{eq:QP_probl_primal}.

Next, we introduce the algorithm in \cite{badoiu2008optimal} that is proven to
compute a core-set of dimension $\dcore$ for the minimum enclosing ball problem
\eqref{eq:QP_probl_primal}. 

Given a set of points $P$, the algorithm can be initialized by choosing any
subset $\CS \subset P$ of $\dcore$ points. Then the algorithm evolves
as follows:
\begin{itemize}
  \item select a point $a$ of $P$ farthest from the center of the minimum
    enclosing ball of $\CS$;
    \item let $\CS_a = \CS \union \{a\}$;
    \item remove a point $b \in \CS_a$ so that the minimum enclosing ball of the
      set $\CS_a \setminus \{b\}$ is the one with largest radius;
    \item if the new radius is equal to the radius of minimum enclosing ball of
      $\CS$, then return $\CS$. Otherwise set $\CS = \CS_a \setminus \{b\}$ and
      repeat the procedure.
\end{itemize}

More formally, the routing is described in the following table. As before, we
let $\cent(P)$ and $\rad(P)$ be respectively the center and radius of the
minimum enclosing ball containing all the points in a set of points $P$.

\begin{center}
\begin{minipage}[c]{.9\textwidth}
\textbf{function} $\coreset(P, \CS)$
\begin{algorithmic}[1]
\STATE $a = \argmax_{p\in P} \| p - \cent(\CS) \| $
\STATE $\CS_a = \CS \union \{a\}$
\STATE $b = \argmax_{p \in \CS_a} \rad(\CS_a \setminus \{p\})$
\STATE \textbf{if} $\rad(\CS_a\setminus \{b\}) > \rad(\CS)$ 
\STATE \quad $\CS = \CS_a \setminus \{b\}$ and go to step 1
\STATE \textbf{else}
\STATE \quad \textbf{return} $\CS$
\STATE \textbf{end if}
\end{algorithmic}
\end{minipage}
\end{center}

It is worth pointing out once more that if $P=S$, with $S$ the one in
\eqref{eq:QP_probl_vec}, then the $\coreset$ algorithm finds an $\epsilon$-core-set
for \eqref{eq:QP_probl_vec} (or equivalently \eqref{eq:QP_probl_primal}).

The $\coreset$ algorithm will be the local routine implemented in the distributed
optimization algorithm we propose in this paper. That is, each node will use the
algorithm to solve a (smaller) local version of the main problem.   

Next we provide a lemma that states the results in \cite{badoiu2008optimal} by
formally itemizing the properties of the algorithm that we will need in our
distributed optimization algorithm.

\begin{lemma}[\cite{badoiu2008optimal}]
\label{lem:Badoiu}
  Let $P\subset\real^d$ be any point set in $\real^d$. Then 
  \begin{enumerate}
  \item $P$ has an $\epsilon$-core-set of size at most $\dcore$;
  \item the $\coreset$ algorithm computes an $\epsilon$-core-set for $P$ in a finite
    number of iterations;
  \item for any $G\subset P$ the radius of $\coreset(G, \CS)$ is larger than or equal to the radius
    of $\CS$.
  \end{enumerate}
\end{lemma}
\begin{proof}
  Statements (i) and (ii) are proven in \cite[Theorem 3.5]{badoiu2008optimal},
  while (iii) follows immediately by step~4 of the algorithm.
\end{proof}

\subsection{Core-set consensus algorithm description}
\label{subsec:algorithm_description}
Let $S$ be the matrix characterizing problem \eqref{eq:QP_probl_vec} or
consistently the set of vectors $s_i\in\real^d$, $i\in\until{N}$.
As stated in Section~\ref{sec:distrib_optim_framework}, each node is assigned
one input vector. This assumption is just for clarity of presentation and can be
easily removed. In fact, the algorithm can be run even if each node is assigned
more than one vector. For this reason we denote $S_i$ the set of initial
vectors, so that under the above assumption we have $S_i = \{s_i\}$.

An informal description of the \emph{core-set consensus} distributed algorithm
is the following.
Each node stores a candidate core-set $\CS_i$, i.e., a set of $\dcore$ vectors
that represent node-$i$'s current estimate for the core-set of $S$. At each
communication round each node receives the candidate core sets (sets of $\dcore$
vectors) from its in-neighbors and initializes its local routine to the core-set
with highest value. Let $\subscr{S}{tmp} = S_i \union \CS_i \union \big(
\union_{j\in \innbrs_i(t)} \CS_j \big)$ be the set of vectors from all
neighboring core-sets plus the initial vectors assigned to node $i$. The local
routine at each node finds a core set (of $\dcore$ vectors) of
$\subscr{S}{tmp}$, and updates the candidate core set with the returned value.

A pseudo-code of the algorithm is given in the following table.
We assume each node can run two routines, namely $\subscr{C}{out} =
\coreset(G,\subscr{C}{in})$ and $\subscr{\val}{out} = \val(\subscr{C}{in})$ returning
respectively the core-set of a given set of vectors $G$ (the routine is
initialized with $\subscr{\CS}{in}$) and the value of a given core set, i.e.,
the optimal value of problem \eqref{eq:QP_probl_vec} with matrix $S=G$ (squared
radius of the minimum enclosing ball).

\begin{center}
\begin{minipage}[c]{0.85\linewidth}
  \begin{algorithm}[H]
    \caption{Core-set consensus}
    \label{alg:core-set_consensus}
    \begin{algorithmic}
      \STATE {\bf Processor state:} $\CS_i \in \real^{d\times\dcore}$
      \STATE {\bf Initialization:} $S_i = \{s_i\}$, $\CS_i = \{s_i, \ldots, s_i\}$  
      \STATE {\bf Message to out-neighbors:} $\CS_i$
      \STATE {\bf Local routine:}
      \begin{align*}
        \subscr{S}{tmp} &:= S_i \union \CS_i \union \big( \union_{j\in
          \innbrs_i(t)} \CS_j \big)\\
        \CS_{i0} &:= \argmax_{j\in\innbrs_i(t)\union \{i\}} \val(\CS_j)\\
        \CS_i &:= \coreset( \subscr{S}{tmp}, \CS_{i0})
      \end{align*}
      \STATE \textbf{return} $\CS_i$
    \end{algorithmic}
  \end{algorithm}
\end{minipage}
\end{center}

\medskip
\begin{remark}
  The algorithm works also if a larger set of vectors is assigned to each
  node. Only the initialization needs to be changed. If a node is assigned more
  than $\dcore$ vectors, it will initialize $\CS_i$ with a random set of
  $\dcore$ vectors. \oprocend
\end{remark}

\begin{remark}
  It is worth noting that the nodes need to know the common tolerance $\epsilon$
  (and thus the core-set dimension $\dcore$) to run the algorithm. \oprocend 
\end{remark}

To analyze the algorithm, we associate a universal, discrete time
$t\in\naturalzero$ to each step of the distributed algorithm evolution, i.e., to
each computation and communication round. This time $t$ is the one defining the
time varying nature of the communication graph and, thus, the same used in
Assumption~\ref{ass:PeriodicConnectivity}.

\subsection{Algorithm analysis}
\label{subsec:algorithm_analysis}
We are now ready to analyze the convergence properties of the algorithm.

\begin{assumption}[Non-degeneracy]
  \label{ass:non_degenerate}
  Given $S$ in \eqref{eq:QP_probl_vec}, for any $C_1, C_2\subset S$ with
  $C_1,C_2\in\real^{d\times\dcore}$, then $\rad(C_1) \neq \rad(C_2)$. \oprocend
\end{assumption}

The above assumption can be removed by using a total ordering for the choice of
$\CS_{i0}$ in Algorithm~\ref{alg:core-set_consensus}. For example, if two
candidate sets $\CS_{i0}^1$ and $\CS_{i0}^2$ have $\rad(\CS_{i0}^1) =
\rad(\CS_{i0}^2)$, then one of the two could be uniquely chosen by using a
lexicographic ordering on the vectors.

\begin{theorem}
  Consider a network of processors with set of identifiers $\Vc = \until{N}$ and
  communication graph $\Gc(t)=(\Vc,\Ec(t))$, $t \in \mathbb{N}$, satisfying
  Assumption~\ref{ass:PeriodicConnectivity}. Suppose
  problem~\eqref{eq:QP_probl_primal} satisfies
  Assumption~\ref{ass:non_degenerate} and has a minimum value $r_*^2$.
  Then the core-set consensus algorithm (Algorithm~\ref{alg:core-set_consensus})
  computes an $\epsilon$-core-set for problem~\eqref{eq:QP_probl_primal} in a
  finite-number of communication rounds. That is, there exists $T>0$ such that
\begin{enumerate}
\item $\CS_i(t) = \bar{\CS}$ for all
  $i\in\until{N}$ and for all $t\geq T$;
\item $\|\cent(\bar{\CS}) - s_i\|^2 \leq (1+\epsilon)^2 r_*^2$ for all
  $i\in\until{N}$.
\end{enumerate}
\end{theorem}
\begin{proof}
  We prove the statement in three steps. First, we prove that each core set
  converges in a finite-number of communication rounds to a stationary set of
  vectors. Second, we prove that (due to Assumption~\ref{ass:non_degenerate})
  all the stationary core sets are equal. Third and finally, we prove that the
  common steady-state set is a core set for problem~\eqref{eq:QP_probl_primal}.

  To prove the first part, notice that by the choice of $\CS_{i0}$ in
  Algorithm~\ref{alg:core-set_consensus} and by Lemma~\ref{lem:Badoiu},
  $\val(\CS_i(t))$ is a monotone nondecreasing function for all $i\in\until{N}$
  along the algorithm evolution. 
  Thus, due to the finite possible values that $\CS_i$ can assume (it is a set
  of $\dcore$ vectors out of $N$ vectors), $\val(\CS_i(t))$ converges to a
  stationary value in finite time.

  To prove the second part, suppose that at some time $T>0$ all the
  $\val(\CS_i(t))$s have converged to a stationary value and that there exist at
  least two nodes $i$ and $j$ such that $\val(\CS_i(t)) >
  \val(\CS_j(t))$. Without loss of generality, from
  Assumption~\ref{ass:PeriodicConnectivity}, we can choose the two nodes so that
  $(i,j)\in\Ec(\bar{t})$, i.e., $(i,j)$ is an edge in $\Gc(\bar{t})$ for some
  time instant $\bar{t} > T$. But from Algorithm~\ref{alg:core-set_consensus} at
  time $\bar{t}$ node $j$ would choose $\CS_i$ to initialize its $\coreset$
  routine, thus leading to a contradiction. 
  From Assumption~\ref{ass:non_degenerate}, it follows $\CS_i = \bar{\CS}$ for
  all $i\in\until{N}$.

  Finally, we just need to prove that $\bar{\CS}$ is a core-set for $S$. But
  from the properties of the $\coreset$ algorithm, for each node
  $i\in\until{N}$, $\CS_i = \bar{\CS}$ is a core-set for the a set of points
  including $s_i$, so that $\bar{\CS}$ is a core set for $S = [s_1 \ldots s_N]$,
  thus concluding the proof.
\end{proof}

\begin{remark}[Core-sets and active constraints]
  A core-set $\CS$ is a set of ``active constraints'' in
  problem~\eqref{eq:QP_probl_primal} with a cost $\val(\CS)$ (i.e., $r =
  \rad(\CS)$). Clearly, some of the constraints will be
  violated for this value of $r$, but no one will be violated for $r = r_*
  (1+\epsilon)$, with $r_*$ being the optimal value of $r$. 
  An equivalent characterization for the core-set is that no constraint is
  violated if $\rad(\CS)$ is relaxed to $\rad(\CS)(1-\epsilon)$. This test is
  easier to run, since it does not involve the computation of the optimal
  value and will be used in the simulations. \oprocend
\end{remark}

\section{Simulations}
\label{sec:simulations}
In this section we provide a numerical example showing the effectiveness of the
proposed strategy. 

We consider a network with $N=100$ nodes communicating according to a directed,
time-varying graph obtained by extracting at each time-instant an
Erd\H{o}s-R\'enyi graph with parameter $0.01$. We choose a small value, so that
at a given instant the graph is disconnected with high probability, but the
graph turns out to be jointly connected. 
We solve a quadratic program, \eqref{eq:QP_probl_primal}, with $d=50$ and choose
a tolerance $\epsilon = 0.1$ so that the number of vectors in the core-set is
$\dcore = 10$. 

In Figure~\ref{fig:rr_transient} and Figure~\ref{fig:cc_transient} the evolution
of the squared-radius and center-norm of the core-sets at each node are
depicted. As expected from the theoretical analysis, the convergence of the
radius to the consensus value is monotone non-decreasing.

\begin{figure}[h]
\centering
\includegraphics[scale=0.5]{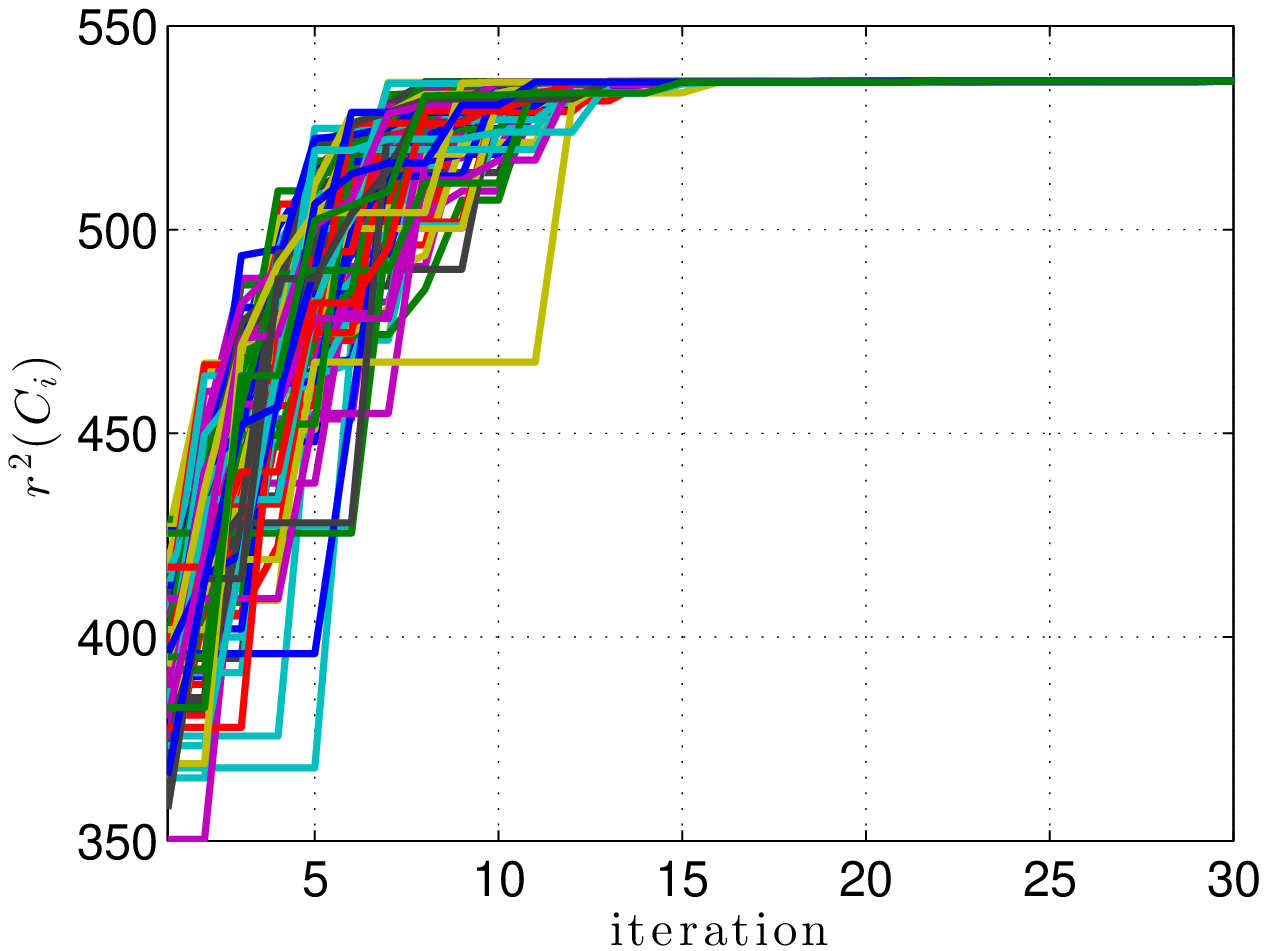}
\caption{Time evolution of $\rad^2(\CS_i(t))$, $i\in\until{N}$.}%
\label{fig:rr_transient}
\end{figure}
\begin{figure}[h]
\centering
\includegraphics[scale=0.5]{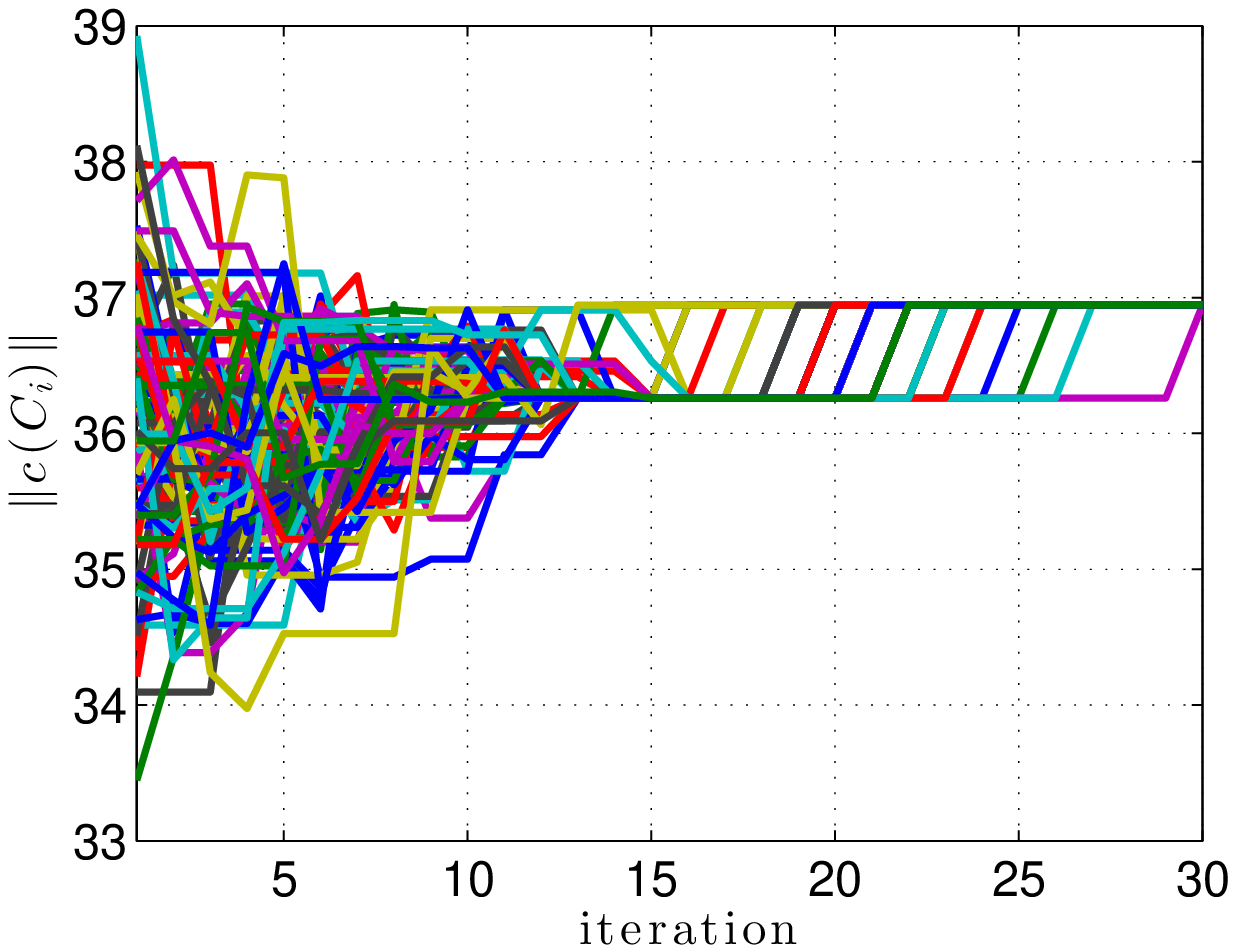}
\caption{Time evolution of $\|\cent(\CS_i(t))\|$, $i\in\until{N}$.}%
\label{fig:cc_transient}
\end{figure}

\section{Conclusions}
In this paper we have proposed a distributed algorithm to solve a special class
of quadratic programs that models several classification problems. The proposed
algorithm handles problems in which not only the number of input data is large,
but furthermore their dimension is big. The resulting learning area is known as
\emph{big-data classification}. We have proposed a distributed optimization
algorithm that computes an approximate solution of the global
problem. Specifically, for any chosen tolerance $\epsilon$, each local node
needs to store only $\dcore$ active constraints, which represent a solution for
the global quadratic program up to a relative tolerance $\epsilon$.
Future research developments include the extension of the algorithmic idea,
based on core-sets, to other big-data optimization problems.


\begin{thebibliography}{10}
\providecommand{\url}[1]{#1}
\csname url@samestyle\endcsname
\providecommand{\newblock}{\relax}
\providecommand{\bibinfo}[2]{#2}
\providecommand{\BIBentrySTDinterwordspacing}{\spaceskip=0pt\relax}
\providecommand{\BIBentryALTinterwordstretchfactor}{4}
\providecommand{\BIBentryALTinterwordspacing}{\spaceskip=\fontdimen2\font plus
\BIBentryALTinterwordstretchfactor\fontdimen3\font minus
  \fontdimen4\font\relax}
\providecommand{\BIBforeignlanguage}[2]{{%
\expandafter\ifx\csname l@#1\endcsname\relax
\typeout{** WARNING: IEEEtran.bst: No hyphenation pattern has been}%
\typeout{** loaded for the language `#1'. Using the pattern for}%
\typeout{** the default language instead.}%
\else
\language=\csname l@#1\endcsname
\fi
#2}}
\providecommand{\BIBdecl}{\relax}
\BIBdecl

\bibitem{cevher2014convex}
V.~Cevher, S.~Becker, and M.~Schmidt, ``Convex optimization for big data:
  Scalable, randomized, and parallel algorithms for big data analytics,''
  \emph{IEEE Signal Processing Magazine}, vol.~31, no.~5, pp. 32--43, 2014.

\bibitem{facchinei2015parallel}
F.~Facchinei, G.~Scutari, and S.~Sagratella, ``Parallel selective algorithms
  for nonconvex big data optimization,'' \emph{IEEE Transactions on Signal
  Processing}, vol.~63, no.~7, pp. 1874--1889, 2015.

\bibitem{slavakis2014online}
K.~Slavakis and G.~B. Giannakis, ``Online dictionary learning from big data
  using accelerated stochastic approximation algorithms,'' in \emph{2014 IEEE
  International Conference on Acoustics, Speech and Signal Processing
  (ICASSP)}, 2014, pp. 16--20.

\bibitem{lu2008distributed}
Y.~Lu, V.~Roychowdhury, and L.~Vandenberghe, ``Distributed parallel support
  vector machines in strongly connected networks,'' \emph{IEEE Transactions on
  Neural Networks}, vol.~19, no.~7, pp. 1167--1178, 2008.

\bibitem{forero2010consensus}
P.~A. Forero, A.~Cano, and G.~B. Giannakis, ``Consensus-based distributed
  support vector machines,'' \emph{Journal of Machine Learning Research},
  vol.~11, pp. 1663--1707, 2010.

\bibitem{lee2012drsvm}
S.~Lee and A.~Nedic, ``Drsvm: Distributed random projection algorithms for
  svms,'' in \emph{{IEEE} Conf.\ on Decision and Control}, 2012, pp.
  5286--5291.

\bibitem{boyd2011distributed}
S.~Boyd, N.~Parikh, E.~Chu, B.~Peleato, and J.~Eckstein, ``Distributed
  optimization and statistical learning via the alternating direction method of
  multipliers,'' \emph{Foundations and Trends{\textregistered} in Machine
  Learning}, vol.~3, no.~1, pp. 1--122, 2011.

\bibitem{varagnolo2013finding}
D.~Varagnolo, S.~Del~Favero, F.~Dinuzzo, L.~Schenato, and G.~Pillonetto,
  ``Finding potential support vectors in separable classification problems,''
  \emph{IEEE Transactions on Neural Networks and Learning Systems}, vol.~24,
  no.~11, pp. 1799--1813, 2013.

\bibitem{notarstefano2011distributed}
G.~Notarstefano and F.~Bullo, ``Distributed abstract optimization via
  constraints consensus: Theory and applications,'' \emph{IEEE Transactions on
  Automatic Control}, vol.~56, no.~10, pp. 2247--2261, October 2011.

\bibitem{badoiu2002approximate}
M.~B{\=a}doiu, S.~Har-Peled, and P.~Indyk, ``Approximate clustering via
  core-sets,'' in \emph{Proceedings of the thiry-fourth annual ACM symposium on
  Theory of computing}, 2002, pp. 250--257.

\bibitem{badoiu2008optimal}
M.~Badoiu and K.~L. Clarkson, ``Optimal core-sets for balls,'' \emph{Comput.
  Geom.}, vol.~40, no.~1, pp. 14--22, 2008.

\bibitem{tsang2005core}
I.~W. Tsang, J.~T. Kwok, and P.-M. Cheung, ``Core vector machines: Fast svm
  training on very large data sets,'' in \emph{Journal of Machine Learning
  Research}, 2005, pp. 363--392.

\bibitem{lodi2010single}
S.~Lodi, R.~Nanculef, and C.~Sartori, ``Single-pass distributed learning of
  multi-class svms using core-sets,'' in \emph{Proceedings of the 2010 SIAM
  International Conference on Data Mining}, 2010, pp. 257--268.

\bibitem{keerthi2000fast}
S.~S. Keerthi, S.~K. Shevade, C.~Bhattacharyya, and K.~R. Murthy, ``A fast
  iterative nearest point algorithm for support vector machine classifier
  design,'' \emph{IEEE Transactions on Neural Networks}, vol.~11, no.~1, pp.
  124--136, 2000.

\bibitem{CC-VV:95}
C.~Cortes and V.~Vapnik, ``Support-vector networks,'' \emph{Machine Learning},
  vol.~20, pp. 273--297, 1995.

\end{thebibliography}
\end{document}